\documentclass[reqno, 12pt]{amsart}

\pdfoutput=1

\usepackage{enumerate}
\usepackage{latexsym}
\usepackage[centertags]{amsmath}
\usepackage{amsfonts}
\usepackage{amssymb}
\usepackage{amsthm}
\usepackage{newlfont}
\usepackage{graphics}
\usepackage{color}
\usepackage{float}
\usepackage{diagbox}
\usepackage{longtable}
\usepackage{rotating}
\usepackage{multirow}

\usepackage{extarrows}

\usepackage[sort,compress,numbers]{natbib}

\usepackage[utf8]{inputenc}

\newtheorem{thm}{Theorem}[section]
\newtheorem{cor}[thm]{Corollary}
\newtheorem{lem}[thm]{Lemma}
\newtheorem{prop}[thm]{Proposition}

\theoremstyle{mydefinition}

\theoremstyle{myremark}
\newtheorem{rem}[thm]{Remark}

\allowdisplaybreaks[4]

\title{A Note on Generalized Repunit Numerical Semigroups}

\author{Feihu Liu$^{1}$, Guoce Xin$^{2, *}$, Suting Ye$^{3}$ and Jingjing Yin$^{4}$}

\address{$^{1, 2, 3, 4}$School of Mathematical Sciences,  Capital Normal University,
 Beijing 100048,  PR China}
\email{$^1$\texttt{liufeihu7476@163.com}\ \& $^2$\texttt{guoce\_xin@163.com}\ \& $^3$\texttt{yesuting0203@163.com}\ \newline \newline \& $^4$\texttt{yinjingj@163.com}}
\date{June 19, 2023}
\thanks{$*$ This work was partially supported by NSFC(12071311).}
\begin{document}
\maketitle

\begin{abstract}
Let $A=(a_1, a_2, ..., a_n)$ be relative prime positive integers with $a_i\geq 2$. The Frobenius number $F(A)$ is the largest integer not belonging to the numerical semigroup $\langle A\rangle$ generated by $A$. The genus $g(A)$ is the number of positive integer elements that are not in $\langle A\rangle$. The Frobenius problem is to find $F(A)$ and $g(A)$ for a given sequence $A$. In this note, we study the Frobenius problem of $A=\left(a,ba+d,b^2a+\frac{b^2-1}{b-1}d,...,b^ka+\frac{b^k-1}{b-1}d\right)$ and obtain formulas for $F(A)$ and $g(A)$ when $a\geq k-1$. Our formulas simplifies further for some special cases, such as repunit, Mersenne and Thabit numerical semigroups. The idea is similar to that in [\cite{LiuXin23},arXiv:2306.03459].
\end{abstract}

\def\D{{\mathcal{D}}}

\noindent
\begin{small}
 \emph{Mathematic subject classification}: Primary 11D07; Secondary 11A67, 11B75, 20M14.
\end{small}

\noindent
\begin{small}
\end{small}

\section{Introduction}
In \cite{LiuXin23}, we studied the Frobenius problem for the numerical semigroup $\langle A\rangle=\langle a,2a+d,2^2a+3d,...,2^ka+(2^k-1)d\rangle$.
We show in this note that the same idea applies to the Frobenius problem for the numerical semigroup $\langle A\rangle=\left\langle a,ba+d,b^2a+\frac{b^2-1}{b-1}d,...,b^ka+\frac{b^k-1}{b-1}d\right\rangle$.
Let $b=2$, we can obtain the results in \cite{LiuXin23}. The material in this note is better merged with \cite{LiuXin23}, but that paper has been submitted.

Throughout this paper, $\mathbb{Z}$, $\mathbb{N}$ and $\mathbb{P}$ denote the set of all integers, non-negative integers and positive integers, respectively. We need to introduce some basic definitions in \cite{J.C.Rosales}.

A subset $S$ is a \emph{submonoid} of $\mathbb{N}$ if $S\subseteq \mathbb{N}$, $0\in S$ and $S$ is closed under the addition in $\mathbb{N}$. If $\mathbb{N}\setminus S$ is finite, then we say that $S$ is a \emph{numerical semigroup}. For a sequence (or set) $A=(a_1,a_2,...,a_n)$ with $a_i \in\mathbb{P}$, we denoted by $\langle A\rangle$ the submonoid of $\mathbb{N}$ generated by $A$, that is,
$$\langle A\rangle=\left\{\sum_{i=1}^n x_ia_i \mid n\in \mathbb{P}, x_i\in \mathbb{N}, 1\leq i\leq n \right\}.$$
If $\gcd(A)=1$, then $\langle A\rangle$ is a numerical semigroup (see \cite{J.C.Rosales}). If $\langle A\rangle$ is a numerical semigroup, then we say that $A$ is a \emph{system of generators} of $\langle A\rangle$. If no proper subsequence (or subset) of $A$ generates $\langle A\rangle$, then we say that $A$ is a \emph{minimal system of generators} of $\langle A\rangle$, and the number of $A$ is the embedding dimension of $\langle A\rangle$, denoted by $e(A)$. Moreover, the minimal system of generators is finite and unique (see \cite{J.C.Rosales}).

If $\gcd(A)=1$, then we have the following definitions:
\begin{enumerate}
  \item \emph{The Frobenius number} $F(A)$: The greatest integer not belonging to $\langle A\rangle$.

  \item \emph{The genus (or Sylvester number)} $g(A)$: The cardinality of $\mathbb{N}\backslash \langle A\rangle$.

  \item \emph{The pseudo-Frobenius numbers} $u$: If $u \in \mathbb{Z}\backslash \langle A\rangle$ and $u+s\in \langle A\rangle$ for all $s \in \langle A\rangle\backslash \{0\}$.

  \item The set $PF(A)$: The set of pseudo-Frobenius numbers of $\langle A\rangle$ (see \cite{Rosales2002}).

  \item \emph{The type} $t(A)$ of $\langle A\rangle$: The cardinality of $PF(A)$.
\end{enumerate}
For more knowledge about numerical semigroup, see \cite{A.Assi,J.C.Rosales}.

The Frobenius number $F(A)$ has been widely studied. For $A=(a_1,a_2)$, Sylvester \cite{J. J. Sylvester1} obtained $F(A)=a_1a_2-a_1-a_2$ in 1882. For $n\geq 3$, F. Curtis \cite{F.Curtis} proved that the $F(A)$ can not be given by closed formulas of a certain type. However, many special cases have been studied such as arithmetic progression in \cite{A. Brauer,Roberts1,E. S. Selmer}, geometric sequences in \cite{D.C.Ong}, triangular and tetrahedral sequences in \cite{AMRobles}.
For more special sequences, see \cite{Ramrez Alfonsn,A. Tripathi1,A. Tripathi3,Liu-Xin,Fliuxin22,Rodseth}. Many special numerical semigroups are also considered, such as Fibonacci in \cite{J.M.Marin}, Mersenne in \cite{Rosales2016}, repunit in \cite{Rosales.Repunit}, squares and cubes in \cite{M.Lepilov}, Thabit in \cite{Rosales2015} and other numerical semigroups in  \cite{GuZe2020,KyunghwanSong2020,KyunghwanSong}.

The motivation of this paper comes from the repunit numerical semigroup: $\langle\{ \frac{b^{n+i}-1}{b-1}\mid i\in \mathbb{N}\}\rangle$. Essentially, its minimal system of generators is $S(b,n)=\left(\frac{b^{n}-1}{b-1}, \frac{b^{n+1}-1}{b-1},..., \frac{b^{2n-1}-1}{b-1}\right)$ (\cite{Rosales.Repunit}), that is the embedding dimension $e(S(b,n))=n$. In this paper, we study the following more general case
$$A=(a,Ha+dB)=\left(a,ba+d,b^2a+\frac{b^2-1}{b-1}d,...,b^ka+\frac{b^k-1}{b-1}d\right),$$
where $a,d,b,k\in\mathbb{P}, b\geq 2$ and $\gcd(a,d)=1$.
Here we have the embedding dimension $e(A)\leq k+1$. We can observe that
\begin{enumerate}
\item If $a=2^n-1$, $b=2$, $d=1$, $k=n-1$, then $\langle A\rangle$ is the Mersenne numerical semigroup $S(n)$ in \cite{Rosales2016}.

\item If $a=3\cdot 2^n-1$, $b=2$, $d=1$, $k=n+1$, then $\langle A\rangle$ is the Thabit numerical semigroup $T(n)$ in \cite{Rosales2015}.

\item If $a=(2^m-1)\cdot 2^n-1$, $b=2$, $d=1$, $k=n+m-1$, then $\langle A\rangle$ is a class of numerical semigroups in \cite{GuZeTang}.

\item If $a=(2^m+1)\cdot 2^n-(2^m-1)$, $b=2$, $d=2^m-1$, $k\in\{n+1,n+m-1,n+m\}$, then $\langle A\rangle$ is a class of numerical semigroups in \cite{KyunghwanSong}.

\item If $a=m(2^k-1)+2^{k-1}-1$, $b=2$, $k\geq 3$, then $\langle A\rangle$ is a class of numerical semigroups in \cite{LiuXin23}.

\item If $a=\frac{b^n-1}{b-1}$, $b\geq 2$, $d=1$, $k=n-1$, then $\langle A\rangle$ is the repunit numerical semigroups $S(b,n)$ in \cite{Rosales.Repunit}.

\item If $a=b^{n+1}+\frac{b^n-1}{b-1}$, $b\geq 2$, $d=1$, $k=n+1$, then $\langle A\rangle$ is a class of numerical semigroups in \cite{GuZe2020}.

\item If $a=(b+1)b^n-1$, $b\geq 2$, $d=b-1$, $k=n+1$, then $\langle A\rangle$ is the Thabit numerical semigroup $T_{b,1}(n)$ of the first kind base $b$ in \cite{KyunghwanSong2020}.
\end{enumerate}

The main purpose of this paper is to solve the Frobenius problem of $A=A=(a,Ha+dB)=\left(a,ba+d,b^2a+\frac{b^2-1}{b-1}d,...,b^ka+\frac{b^k-1}{b-1}d\right)$.
Using similar discussions as in paper \cite{LiuXin23}, we can obtain formulas of $F(A)$ and $g(A)$ in Theorem \ref{a-2ad-22a3d}. Specifically, if $k=n-1$ and $a=\frac{b^n-1}{b-1}$, then we can further obtain
\begin{enumerate}
  \item The Frobenius number $F(A)=(b^n+d-1)\cdot\frac{b^n-1}{b-1}-d$.

  \item The genus $g(A)=\frac{(b^n-b)(b^n+d-1)}{2(b-1)}+\frac{b^n(n-1)}{2}$.

  \item The set of Pseudo-Frobenius number $PF(A)=\{F(A), F(A)-d,...,F(A)-(n-2)d\}$.

  \item The type $t(A)=n-1$.
\end{enumerate}

The paper is organized as follows.
In Section 2, we provide some necessary lemmas and related results.
In Section 3, we establish Theorem \ref{a-2ad-22a3d}, which gives formulas of Frobenius number $F(A)$ and genus $g(A)$ where $\left(a,ba+d,b^2a+\frac{b^2-1}{b-1}d,...,b^ka+\frac{b^k-1}{b-1}d\right)$ and $a\ge k-1$.
Section 4 focus on application of Theorem \ref{a-2ad-22a3d} for special $a$ and $k$. It is related to repunit numerical semigroup.
Section 6 is a concluding remark.

\section{Preliminary}

It is convenient to use the short hand notation
$A:=(a, B)=(a, b_1, b_2, ..., b_k)$. Let $\gcd(A)=1$, $a,b_i\in \mathbb{P}$. The set
$$\langle A\rangle=\left\{ ax+\sum_{i=1}^kb_ix_i\ \mid x, x_i\in \mathbb{N}\right\}$$
is a numerical semigroup. Let $w \in \langle A\rangle\backslash \{0\}$. The \emph{Ap\'ery set} of $w$ in $\langle A\rangle$ is $Ape(A,w)=\{s\in \langle A\rangle \mid s-w\notin \langle A\rangle\}$. In \cite{J.C.Rosales}, we can obtain
$$Ape(A,w)=\{N_0,N_1,N_2,...,N_{w-1}\},$$
where $N_r:=\min\{ a_0\mid a_0\equiv r\mod w, \ a_0\in \langle A\rangle\}$, $0\leq r\leq w-1$. We usually take $w:=a$.

A. Brauer and J. E. Shockley \cite{J. E. Shockley}, E. S. Selmer \cite{E. S. Selmer} gave the following results respectively.
\begin{lem}[\cite{J. E. Shockley}, \cite{E. S. Selmer}]\label{LiuXin001}
Suppsoe $A:=(a, B)=(a, b_1, b_2, ..., b_k)$. The \emph{Ap\'ery set} of $a$ in $\langle A\rangle$ is
$Ape(A,a)=\{N_0,N_1,N_2,...,N_{a-1}\}$. Then the Frobenius number and genus of $A$ are respectively:
\begin{align*}
F(A)&=F(a, B)=\max_{r\in \lbrace 0, 1, ..., a-1\rbrace}N_r -a,\\
g(A)&=g(a, B)=\frac{1}{a}\sum_{r=1}^{a-1}N_r-\frac{a-1}{2}.
\end{align*}
\end{lem}

Now we define the following order relation in $\mathbb{Z}$: $a\preceq_{\langle A\rangle} b$ if $b-a \in \langle A\rangle$. It is proved that this relation $\preceq_{\langle A\rangle}$ is an order relation in \cite{J.C.Rosales}.

\begin{lem}[Proposition 2.20, \cite{J.C.Rosales}]\label{Pseudo-FP-Prop}
Let $\langle A\rangle$ be a numerical semigroup. Then
$$PF(A)=\left\{w-a\mid w \in \max\nolimits_{\preceq_{\langle A\rangle}} Ape(A, a) \right\},$$
where $Ape(A,a)=\{N_0,N_1,...,N_{a-1}\}$.
\end{lem}

Note that $N_0=0$ for all $A$. By the definition of $N_r$, we can easily obtain the following result.
\begin{prop}\label{0201}
Let $A=(a, b_1, ..., b_k)$,  $\gcd(a, d)=1, \ d\in \mathbb{P}$. Then we have
\begin{equation}
\{N_0, N_1, N_2,..., N_{a-1}\}=\{N_{d\cdot 0}, N_{d\cdot 1}, N_{d\cdot 2},..., N_{d\cdot (a-1)}\}.\label{0402}
\end{equation}
\end{prop}

\begin{lem}[\cite{LiuXin23}]\label{0202}
Let $A=(a, h_1a+db_1, ..., h_ka+db_k)$,  $k, h, d\in\mathbb{P}$ and $\gcd(A)=1$, $m\in\mathbb{N}$,  $\gcd(a, d)=1$. For a given $ 0\leq r\leq a-1$,  we have
\begin{equation}\label{0203}
N_{dr}=\min \left\{O_{B}^{H}(ma+r) \cdot a+(ma+r)d \mid m\in \mathbb{N}\right\},
\end{equation}
where
$$O_{B}^{H}(M):=\min\left\{\sum_{i=1}^kh_ix_i \mid \sum_{i=1}^k b_ix_i=M, \ x_i,M\in\mathbb{N}, 1\leq i\leq k\right\}.$$
\end{lem}

By Lemma \ref{0202}, we can define an intermediate function with respect to $m\in \mathbb{N}$, namely:
$$N_{dr}(m):=O_{B}^{H}(ma+r) \cdot a+(ma+r)d.$$

We also need the following definition.
For a given positive integers sequence $B=(b_1, b_2, ..., b_k)$, $1=b_1<b_2<\cdots <b_k$ and $M\in \mathbb{N}$, let
\begin{equation}\label{hahaha4}
opt_B(M):=\min\left\{\sum_{i=1}^kx_i \mid \sum_{i=1}^kb_ix_i=M, \ \ M,x_i\in \mathbb{N}, 1\leq i\leq k\right\}.
\end{equation}
The problem $opt_B(M)$ is called \emph{the change-making problem} \cite{AnnAdamaszek}. A strategy is called greedy, that is to use as many of the maximum as possible, than as many of the next one as possible, and so on. We denote $grd_B(M)$ the number of elements used in $B$ by the greedy strategy. Then we have $opt_B(M)\leq grd_B(M)$. If the greedy solution is always optimal, i.e., $opt_B(M)=grd_B(M)$ for all $M>0$, then we call the sequence $B$ \emph{orderly}; Otherwise, we call sequence $B$ \emph{non-orderly}.

\begin{lem}[One-Point Theorem, \cite{LJCowen,TCHu,MJMagazine}]\label{one-point}
Suppose $B^{\prime}=(1,b_1,...,b_k)$ is orderly and $b_{k+1}>b_k$. Let $s=\lceil b_{k+1} / b_k\rceil$. Then the sequence $B=(1,b_1,...,b_k,b_{k+1})$ is orderly if and only if $opt_B(sb_k)=grd_B(sb_k)$.
\end{lem}

\section{The Frobenius Problem for $A=\left(a,ba+d,b^2a+\frac{b^2-1}{b-1}d,...,b^ka+\frac{b^k-1}{b-1}d\right)$}

Now we consider the case $A=(a,Ha+dB)=\left(a,ba+d,b^2a+\frac{b^2-1}{b-1}d,...,b^ka+\frac{b^k-1}{b-1}d\right)$, $b\geq 2$, i.e., $H=(b,b^2,...,b^k)$ and $B=\left(1,\frac{b^2-1}{b-1},...,\frac{b^k-1}{b-1}\right)$. We first consider
$$opt_B(M)=\min\left\{\sum_{i=1}^kx_i \mid \sum_{i=1}^k\frac{b^i-1}{b-1}x_i=M, \ \ M,x_i\in \mathbb{N}, 1\leq i\leq k\right\}.$$

By the \emph{One-Point Theorem}, we can obtain the following result.
\begin{lem}\label{B137ordely}
Let $k\in \mathbb{P}$ and $b\geq 2$. Then the sequence $B=\left(1,\frac{b^2-1}{b-1},...,\frac{b^k-1}{b-1}\right)$ is orderly, i.e., $opt_B(M)=grd_B(M)$ for all $M\in \mathbb{P}$.
\end{lem}
\begin{proof}
By induction, we first know that the sequence $(1)$ and $\left(1,\frac{b^2-1}{b-1}\right)$ are orderly sequences. Suppose $\left(1,\frac{b^2-1}{b-1},...,\frac{b^{k-1}-1}{b-1}\right)$ is orderly. We have $s=\left\lceil \frac{b^k-1}{b^{k-1}-1}\right\rceil=\left\lceil \frac{b(b^{k-1}-1)+b-1}{b^{k-1}-1}\right\rceil=b+1$. By $(b+1)\cdot\frac{b^{k-1}-1}{b-1}=\frac{b^k-1}{b-1}+b\cdot\frac{b^{k-2}-1}{b-1}$, we have  $opt_B(s(\frac{b^{k-1}-1}{b-1}))=grd_B(s(\frac{b^{k-1}-1}{b-1}))=b+1$. By Lemma \ref{one-point}, the sequence $B$ is orderly. This completes the proof.
\end{proof}

For a given $M\in \mathbb{P}$ and $B=\left(1,\frac{b^2-1}{b-1},...,\frac{b^k-1}{b-1}\right)$, by $b\cdot \frac{b^i-1}{b-1}+1=\frac{b^{i+1}-1}{b-1}$ and Lemma \ref{B137ordely}, we have the following properties for $x_i$ in $opt_B(M)$.
\begin{enumerate}
  \item $x_k=\left\lfloor \frac{(b-1)M}{b^k-1}\right\rfloor$.

  \item $x_i\in \{0,1,...,b\}$ for every $1\leq i\leq k-1$.

  \item if $2\leq i\leq k-1$ and $x_i=b$, then $x_1=\cdots =x_{i-1}=0$.
\end{enumerate}

If a solution $X=(x_1,x_2,...,x_k)$ of $opt_B(M)$ satisfies the above conditions (1), (2) and (3), we call $X$ the \emph{greedy presentation} of $M$.
We define $R(M)$ to be the set of all greedy presentations of $0,1,2,...,M$. We define a colexicographic order on $R(M)$ as follows:
\begin{align*}
&(x_1^{\prime},x_2^{\prime},...,x_k^{\prime})\preceq (x_1,x_2,...,x_k)
\\ \Longleftrightarrow & x_i^{\prime}=x_i\ \ \text{for any}\ \ i>0\ \ \text{or}
\\& x_j^{\prime}<x_j, x_i^{\prime}=x_i\ \ \text{for a certain}\ \ j>0\ \ \text{and any}\ \ i>j.
\end{align*}
Obviously, the order relation $\preceq$ is a total order on $R(M)$.

\begin{rem}
Lemma \ref{B137ordely} essentially provides the construction process of $X=(x_1,x_2,...,x_k)$ in $R(M)$.
\end{rem}

Now we consider the $O_B^H(M)$ for $B=\left(1,\frac{b^2-1}{b-1},...,\frac{b^k-1}{b-1}\right)$, $H=(b,b^2,...,b^k)$ and any $M\in \mathbb{P}$. We have
\begin{align*}
O_{B}^{H}(M)&=\min\left\{\sum_{i=1}^kb^ix_i \mid \sum_{i=1}^k \frac{b^i-1}{b-1}x_i=M, \ x_i\in\mathbb{N}, 1\leq i\leq k\right\}
\\&=\min\left\{(b-1)M+\sum_{i=1}^kx_i \mid \sum_{i=1}^k \frac{b^i-1}{b-1}x_i=M, \ x_i\in\mathbb{N}, 1\leq i\leq k\right\}.
\end{align*}

\begin{lem}
Let $A=\left(a,ba+d,b^2a+\frac{b^2-1}{b-1}d,...,b^ka+\frac{b^k-1}{b-1}d\right)$, $a,d,b,k\in \mathbb{P}$, $b\geq 2$ and $\gcd(a,d)=1$. Suppose $a\geq k-1$. Then $N_{dr}(m)$ is increasing with respect to $m\in \mathbb{N}$. More precisely, if $X=(x_1,x_2,...,x_k)$ is a greedy presentation of $r$, then
\begin{equation}\label{Ndr-pre}
N_{dr}=\left(\sum_{i=1}^kx_i\right)a+r((b-1)a+d)=\left(\sum_{i=1}^kb^ix_i\right)a+rd.
\end{equation}
\end{lem}
\begin{proof}
Given a $m\in \mathbb{N}$. Assuming $X^{\prime}=(x_1^{\prime},...,x_k^{\prime})$ and $X=(x_1,...,x_k)$ are greedy presentations of $(m+1)a+r$ and $ma+r$, respectively. We can easily get $x_{k}^{\prime}\geq x_k$. Furthermore
$$N_{dr}(m)=\left((b-1)(ma+r)+\sum_{i=1}^kx_i\right)a+(ma+r)d.$$
Therefore we have
\begin{align*}
N_{dr}(m+1)-N_{dr}(m)&=(b-1)a^2+ad+\left(\sum_{i=1}^kx_i^{\prime}-\sum_{i=1}^kx_i\right)a
\\& \geq\left((b-1)a+d+\sum_{i=1}^{k-1}x_i^{\prime}-\sum_{i=1}^{k-1}x_i\right)a
\\& \geq((b-1)a+d-b-(k-2)(b-1))a
\\& =((b-1)(a-k+1)+d-1)a\geq 0.
\end{align*}
Then $N_{dr}(m)$ is increasing with respect to $m\in \mathbb{N}$. Moreover $N_{dr}=N_{dr}(0)$. This completes the proof.
\end{proof}

In Equation \eqref{Ndr-pre}, suppose $X=(x_1,x_2,...,x_k)$ is a greedy presentation of $r$, we call $w(r)=\sum_{i=1}^kb^ix_i$ the \emph{weight} of $r$. Now we have the following result.

\begin{lem}\label{colex-incre}
Suppose $X^{\prime}=(x_1^{\prime},...,x_k^{\prime})\in R(M)$ and $X=(x_1,...,x_k)\in R(M)$ are greedy presentations of $r_1$ and $r_2$, respectively. If $(x_1^{\prime},...,x_k^{\prime})\preceq (x_1,...,x_k)$, then we have $w(r_1)\leq w(r_2)$.
\end{lem}
\begin{proof}
For $X=(x_1,x_2,...,x_k)\in R(M)$, it is easy to obtain that the maximum of $\sum_{i=1}^{k-1}b^ix_i$ is $b^k$. At this point, the possible values for $(x_1,...,x_{k-1})$ are $$(b,b-1,b-1,...,b-1), (0,b,b-1,...,b-1),..., (0,0,0,...,b).$$
We consider the following three cases.

If $x_i^{\prime}=x_i$ for any $1\leq i\leq k$, then $w(r_1)=w(r_2)$.

If $x_k^{\prime}<x_k$, then $x_k\geq x_{k}^{\prime}+1$ and
\begin{align*}
w(r_1)&=\sum_{i=1}^kb^ix_i^{\prime}=\sum_{i=1}^{k-1}b^ix_i^{\prime}+b^kx_k^{\prime}\leq b^k+b^kx_k^{\prime}=b^k(1+x_k^{\prime})\leq b^k x_k\leq w(r_2).
\end{align*}

If there is a certain $1\leq u\leq k-1$, so that $x_u^{\prime}<x_u$ and $x_{u+1}^{\prime}=x_{u+1},...,x_k^{\prime}=x_k$. Similarly we have
\begin{align*}
w(r_1)&=\sum_{i=1}^kb^ix_i^{\prime}=\sum_{i=1}^{u}b^ix_i^{\prime}+\sum_{j=u+1}^{k}b^jx_j^{\prime}
\\&\leq b^u(1+x_u^{\prime})+\sum_{j=u+1}^{k}b^jx_j^{\prime}=b^u(1+x_u^{\prime})+\sum_{j=u+1}^{k}b^jx_j
\\&\leq b^ux_u +\sum_{j=u+1}^{k}b^jx_j\leq w(r_2).
\end{align*}
This completes the proof.
\end{proof}

Now we can obtain the main result of this section.
\begin{thm}\label{a-2ad-22a3d}
Let $A=\left(a,ba+d,b^2a+\frac{b^2-1}{b-1}d,...,b^ka+\frac{b^k-1}{b-1}d\right)$, $a,d,b,k\in \mathbb{P}$, $\gcd(a,d)=1$, $b\geq 2$ and $a\geq k-1$. We have
\begin{align*}
&F(A)=\left((b-1)a-b+d+\left(\sum_{i=1}^kx_i\right)_{a-1}\right)a-d,
\\&g(A)=\sum_{r=1}^{a-1}\left(\sum_{i=1}^kx_i\right)_r+\frac{(a-1)((b-1)a+d-1)}{2},
\end{align*}
where $(\sum_{i=1}^{k}x_i)_r$ is the sum of elements in the greedy presentation of $r$. Obviously $x_k=\left\lfloor \frac{(b-1)r}{b^k-1}\right\rfloor$.
\end{thm}
\begin{proof}
By Lemma \ref{colex-incre} and Equation \eqref{Ndr-pre}, we have
$$\max N_{dr}=N_{d(a-1)}=\left(\sum_{i=1}^kx_i\right)_{a-1}\cdot a+(a-1)((b-1)a+d).$$
Further from Lemma \ref{LiuXin001}, we can obtain the Frobenius number formula $F(A)$. From Equation \eqref{Ndr-pre} and Lemma \ref{LiuXin001} again, we can obtain the genus $g(A)$. This completes the proof.
\end{proof}

\section{A Generalization of Repunit Numerical Semigroup}

In this section we mainly focus on the generalization of repunit numerical semigroup.
For $A=\left(a,ba+d,b^2a+\frac{b^2-1}{b-1}d,...,b^ka+\frac{b^k-1}{b-1}d\right)$, suppose $k:=n-1$, $a:=\frac{b^n-1}{b-1}$, $b\geq 2$. Therefore we have
$$A=\Bigg(\frac{b^n-1}{b-1}, b\cdot \frac{b^n-1}{b-1}+d, b^2\cdot \frac{b^n-1}{b-1}+\frac{b^2-1}{b-1}d, ..., b^{n-1}\cdot \frac{b^n-1}{b-1}+\frac{b^{n-1}-1}{b-1}d\Bigg).$$
We consider the numerical semigroup $\langle A\rangle$ generated by $A$.

By $a-1=\frac{b^n-1}{b-1}-1=b\cdot \frac{b^{n-1}-1}{b-1}$,
we have
\begin{align}
R(a-1)&=\Big\{(x_1,x_2,...,x_{n-1}) \mid \ 0\leq x_i\leq b\ \ \text{for}\ \ 1\leq i\leq n-1;\notag
\\&\ \ \ \ \
\text{if}\ \ x_i=b,\ \ \text{then}\ \  x_j=0\ \  \text{for}\ \ j\leq i-1\Big\}.
\end{align}

\begin{thm}\label{m2nd-Mersen}
For $A=\left(a,ba+d,b^2a+\frac{b^2-1}{b-1}d,...,b^ka+\frac{b^k-1}{b-1}d\right)$, suppose $k:=n-1$, $a:=\frac{b^n-1}{b-1}$, $b,n\geq 2$ and $\gcd(a,d)=1$. We have
\begin{align*}
&F(A)=(b^n+d-1)\cdot\frac{b^n-1}{b-1}-d,
\\&g(A)=\frac{(b^n-b)(b^n+d-1)}{2(b-1)}+\frac{b^n(n-1)}{2}.
\end{align*}
\end{thm}
\begin{proof}
By Theorem \ref{a-2ad-22a3d} and $a-1=b\cdot \frac{b^{n-1}-1}{b-1}$, we have
\begin{align*}
F(A)&=\left((b-1)\cdot \frac{b^n-1}{b-1}-b+d+b\right)\cdot \frac{b^n-1}{b-1}-d
\\&=(b^n+d-1)\cdot\frac{b^n-1}{b-1}-d.
\end{align*}
Now we consider the genus $g(A)$. Based on the composition of $R(a-1)$ mentioned above, we have
\begin{align*}
\sum_{r=1}^{a-1}\left(\sum_{i=1}^{n-1} x_i\right)_r=&(n-1)b^{n-2}\cdot\sum_{i=0}^{b-1}i+\sum_{j=1}^{n-1}\left(b\cdot b^{n-1-j}+(n-1-j)\cdot b^{n-2-j}\cdot\sum_{i=0}^{b-1}i\right)
\\=&\frac{(n-1)b^{n-1}(b-1)}{2}+\sum_{j=1}^{n-1}\left(b^{n-j}+\frac{b^{n-1-j}(b-1)(n-j-1)}{2}\right)
\\=&\frac{b^n-b}{2(b-1)}+\frac{b^n(n-1)}{2}.
\end{align*}
By Theorem \ref{a-2ad-22a3d}, we have
\begin{align*}
g(A)&=\frac{(b^n-b)(b^n+d-1)}{2(b-1)}+\frac{b^n(n-1)}{2}.
\end{align*}
This completes the proof.
\end{proof}

Let $d=1$ in Theorem \ref{m2nd-Mersen}, the $\langle A\rangle$ becomes the repunit numerical semigroup (\cite{Rosales.Repunit}). Clearly, we have the following corollary.

\begin{cor}[Theorem 20, Theorem 25, \cite{Rosales.Repunit}]
Let $n,b\geq 2$, $d=1$ in Theorem \ref{m2nd-Mersen}. Then $S(b,n)=\langle A\rangle$ is the repunit numerical semigroup in \cite{Rosales.Repunit}. Furthermore, we have
\begin{align*}
&F(S(b,n))=\frac{b^n-1}{b-1}b^n-1,
\\&g(S(b,n))=\frac{b^n}{2}\left(\frac{b^n-1}{b-1}+n-1\right).
\end{align*}
\end{cor}

\begin{thm}\label{gener-mers-type}
For $A=\left(a,ba+d,b^2a+\frac{b^2-1}{b-1}d,...,b^ka+\frac{b^k-1}{b-1}d\right)$, suppose $k:=n-1$, $a:=\frac{b^n-1}{b-1}$, $b,n\geq 2$ and $\gcd(a,d)=1$. Then $t(A)=n-1$. Furthermore
$$PF(A)=\{F(A), F(A)-d,...,F(A)-(n-2)d\}.$$
\end{thm}
\begin{proof}
Let $X=(x_1,x_2,...,x_{n-1})$ be a greedy presentation of $r$, $0\leq r\leq a-1$.
By Equation \eqref{Ndr-pre}, we have
\begin{align*}
N_{dr}&=\left(\sum_{i=1}^{n-1}b^ix_i\right)a+\sum_{i=1}^{n-1}\frac{b^i-1}{b-1}x_id
\\&=\left(\sum_{i=1}^{n-1}b^ix_i\right)a+\sum_{i=1}^{n-1}\frac{b^ix_i}{b-1}d
-\sum_{i=1}^{n-1}\frac{x_i}{b-1}d.
\end{align*}
For a given $0\leq r\leq a-1$, we have
$r=s\cdot \frac{b^{n-1}-1}{b-1}+r^{\prime}$, $0\leq s\leq b$ and $0\leq r^{\prime}\leq \frac{b^{n-1}-b}{b-1}$. Therefore, under the order relation $\preceq_{\langle A\rangle}$,  the greedy presentations of the maximal elements in $R(a-1)$ are contained in the following elements:
$$(b,b-1,b-1,...,b-1,b-1),\ (0,b,b-1,...,b-1,b-1),\ (0,0,b,...,b-1,b-1),\ ... ,\ (0,0,0,...,0,b).$$

For the above greedy presentations, we always have
$$\sum_{i=1}^{n-1}b^ix_i=b^n.$$
For the above maximal elements, let $X^{\prime}=(x_1^{\prime},...,x_{n-1}^{\prime})\in R(a-1)$ and $X=(x_1,...,x_{n-1})\in R(a-1)$ be greedy presentations of $r_1$ and $r_2$, respectively.
We need to consider the following case
$$N_{dr_1}-N_{dr_2}=\sum_{i=1}^{n-1}\frac{x_i}{b-1}d-\sum_{i=1}^{n-1}\frac{x_i^{\prime}}{b-1}d=td,\ \ t=1,2,...,n-2.$$
Now we need to determine whether $td$ belongs to $\langle A\rangle$. If $td\in \langle A\rangle$, there would exist $(y_0,y_1,...,y_{n-1})$ such that
$$td=y_0a+y_1(ba+d)+\cdots +y_{n-1}\left(b^{n-1}a+\frac{b^{n-1}-1}{b-1}d\right).$$
These $y_0,...,y_{n-1} \in \mathbb{N}$ are not all $0$.
Then we have
$$(n-2)d\geq \left(t-y_1-y_2\frac{b^2-1}{b-1}-\cdots -y_{n-1}\frac{b^{n-1}-1}{b-1}\right)d=(y_0+by_1+\cdots +b^{n-1}y_{n-1})a>0.$$
By $\gcd(a,d)=1$, we have $a | \left(t-y_1-y_2\frac{b^2-1}{b-1}-\cdots -y_{n-1}\frac{b^{n-1}-1}{b-1}\right)$.
However $a>n-2\geq t$, we have $t=y_1+y_2\frac{b^2-1}{b-1}+\cdots +y_{n-1}\frac{b^{n-1}-1}{b-1}$.
So $(y_0+by_1+\cdots +b^{n-1}y_{n-1})a=0$, a contradiction.

Therefore we have
$$\max\nolimits_{\preceq_{\langle A\rangle}} Ape(A, a)=\{N_{dr} \mid r=a-\{1,2,...,n-1\}\}.$$
By Lemma \ref{Pseudo-FP-Prop}, we have
\begin{align*}
PF(A)&=\{ N_{dr}-a \mid r=a-\{1,2,...,n-1\}\}
\\&=\left\{(b^n-1)a+\frac{b^nd}{b-1}-\left\{\frac{bd}{b-1},\frac{(2b-1)d}{b-1},\frac{(3b-2)d}{b-1},...,
\frac{((n-1)b-(n-2))d}{b-1}\right\} \right\}
\\&=\{F(A),F(A)-d,...,F(A)-(n-2)d\}.
\end{align*}
This completes the proof.
\end{proof}

Similarly, we can obtain the following corollary.
\begin{cor}[Theorem 23, \cite{Rosales.Repunit}]
Let $n,b\geq 2$, $d=1$ in Theorem \ref{m2nd-Mersen}. Then $S(b,n)=\langle A\rangle$ is the repunit numerical semigroup in \cite{Rosales.Repunit}. We have $t(S(b,n))=n-1$ and
\begin{align*}
PF(S(b,n))=\{F(S(n)),F(S(n))-1,...,F(S(n))-(n-2)\}.
\end{align*}
\end{cor}

\section{Concluding Remark}

Our model $A=\left(a,ba+d,b^2a+\frac{b^2-1}{b-1}d,...,b^ka+\frac{b^k-1}{b-1}d\right)$ can be used to explain the following eight numerical semigroups in the literature.

1. If $a=2^n-1$, $b=2$, $d=1$, $k=n-1$ and $n\geq 2$, then
$$A=(2^n-1,2^{n+1}-1,2^{n+2}-1...,2^{2n-1}-1),$$
and the $\langle A\rangle$ is the Mersenne numerical semigroups $S(n)$ in \cite{Rosales2016}.

2. If $a=3\cdot 2^n-1$, $b=2$, $d=1$, $k=n+1$ and $n\geq 1$, then
$$A=(3\cdot 2^n-1, 3\cdot 2^{n+1}-1, 3\cdot 2^{n+2}-1,...,3\cdot 2^{2n+1}-1),$$
and the $\langle A\rangle$ is the Thabit numerical semigroups $T(n)$ in \cite{Rosales2015}.

3. If $a=(2^m-1)\cdot 2^n-1$, $b=2$, $d=1$, $k=n+m-1$, $n\geq 1$ and $2\leq m\leq 2^n$, then
$$A=((2^m-1)\cdot 2^n-1, (2^m-1)\cdot 2^{n+1}-1,...,(2^{m}-1)\cdot 2^{2n+m-1}-1),$$
and the $\langle A\rangle$ is a class of numerical semigroups $S(m,n)$ in \cite{GuZeTang}.

4. Let $m,n\in \mathbb{N}$ and $m\geq 2$ and
$$\begin{aligned}
  \delta=
\left\{
    \begin{array}{lc}
    1&\ \text{if}\ \ n=0;\ \ \ \ \ \ \ \ \ \ \ \\
    m&\ \text{if}\ \ n\neq 0, m\leq n;\\
    m-1&\ \text{if}\ \ n\neq 0, m> n.\\
    \end{array}
\right.
\end{aligned}$$
If $a=(2^m+1)\cdot 2^n-(2^m-1)$, $b=2$, $d=2^m-1$ and $k=n+\delta$, then the $\langle A\rangle$ is a class of numerical semigroups $GT(n,m)$ in \cite{KyunghwanSong}.

5. If $a=m(2^k-1)+2^{k-1}-1$, $b=2$, $m\geq 1$, $k\geq 3$, then the $\langle A\rangle$ is a class of numerical semigroups in \cite{LiuXin23}.

6. If $a=\frac{b^n-1}{b-1}$, $b\geq 2$, $d=1$, $k=n-1$ and $n\geq 2$, then
$$A=\left(\frac{b^n-1}{b-1}, \frac{b^{n+1}-1}{b-1},\frac{b^{n+2}-1}{b-1},...,\frac{b^{2n-1}-1}{b-1}\right),$$
and the $\langle A\rangle$ is the repunit numerical semigroups $S(b,n)$ in \cite{Rosales.Repunit}.

7. If $a=b^{n+1}+\frac{b^n-1}{b-1}$, $b\geq 2$, $d=1$, $k=n+1$, then
$$A=\left(b^{n+1}+\frac{b^n-1}{b-1}, b^{n+2}+\frac{b^{n+1}-1}{b-1},...,b^{2n+2}+\frac{b^{2n+1}-1}{b-1}\right),$$
and the $\langle A\rangle$ is a class of numerical semigroups $S(b,n)$ in \cite{GuZe2020}.

8. If $a=(b+1)b^n-1$, $b\geq 2$, $d=b-1$, $k=n+1$, then
$$A=((b+1)b^n-1,(b+1)b^{n+1}-1,...,(b+1)b^{2n+1}-1),$$
and the $\langle A\rangle$ is the Thabit numerical semigroup $T_{b,1}(n)$ of the first kind base $b$ in \cite{KyunghwanSong2020}.

\noindent
{\small \textbf{Acknowledgements:}
This work was partially supported by the National Natural Science Foundation of China [12071311].

\end{document}